\documentclass{article}%
\usepackage{amssymb}
\usepackage{amsmath}
\usepackage{amsfonts}
\usepackage{graphicx}%
\providecommand{\U}[1]{\protect\rule{.1in}{.1in}}
\newtheorem{theorem}{Theorem}[section]

\newtheorem{corollary}[theorem]{Corollary}

\newtheorem{lemma}[theorem]{Lemma}

\newtheorem{problem}[theorem]{Problem}
\newtheorem{proposition}[theorem]{Proposition}

\newenvironment{proof}[1][Proof]{\noindent\textbf{#1.} }{\ \rule{0.5em}{0.5em}}
\begin{document}

\author{Vadim E. Levit\\Department of Mathematics\\Ariel University, Israel\\levitv@ariel.ac.il
\and Eugen Mandrescu\\Department of Computer Science\\Holon Institute of Technology, Israel\\eugen\_m@hit.ac.il}
\title{Almost Bipartite non-K\"{o}nig-Egerv\'{a}ry Graphs Revisited}
\date{}
\maketitle

\begin{abstract}
Let $\alpha(G)$ denote the cardinality of a maximum independent set, while
$\mu(G)$ be the size of a maximum matching in $G=\left(  V,E\right)  $. It is
known that if $\alpha(G)+\mu(G)=\left\vert V\right\vert $, then $G$ is a
\textit{K\"{o}nig-Egerv\'{a}ry graph }\cite{dem,ster}.

The \textit{critical difference} $d(G)$ is $\max\{d(I):I\in\mathrm{Ind}(G)\}$,
where $\mathrm{Ind}(G)$\ denotes the family of all independent sets of $G$. If
$A\in\mathrm{Ind}(G)$ with $d\left(  X\right)  =d(G)$, then $A$ is a
\textit{critical independent set} \cite{Zhang1990}. For a graph $G$, let
$\mathrm{diadem}(G)=\bigcup\{S:S$ is a critical independent set in $G\}$, and
$\varrho_{v}\left(  G\right)  $ denote the number of vertices $v\in V\left(
G\right)  $, such that $G-v$ is a K\"{o}nig-Egerv\'{a}ry graph. A graph is
called almost bipartite if it has a unique odd cycle.

In this paper we show that if $G$ is an almost bipartite
non-K\"{o}nig-Egerv\'{a}ry graph with the unique odd cycle $C$, then the
following assertions are true:

\begin{itemize}
\item every maximum matching of $G$ contains $\left\lfloor \frac{V(C)}%
{2}\right\rfloor $ edges belonging to $C$;

\item $V(C)\cup N_{G}\left[  \mathrm{diadem}\left(  G\right)  \right]  =V$ and
$V(C)\cap N_{G}\left[  \mathrm{diadem}\left(  G\right)  \right]  =\emptyset$;

\item $\varrho_{v}\left(  G\right)  =\left\vert \mathrm{corona}\left(
G\right)  \right\vert -\left\vert \mathrm{diadem}\left(  G\right)  \right\vert
$, where $\mathrm{corona}\left(  G\right)  $ is the union of all maximum
independent sets of $G$;

\item $\varrho_{v}\left(  G\right)  =\left\vert V\right\vert $ if and only if
$G=C_{2k+1}$ for some integer $k\geq1$.

\end{itemize}


\end{abstract}

\section{Introduction}

Throughout this paper $G=(V,E)$ is a finite, undirected, loopless graph
without multiple edges, with vertex set $V=V(G)$ of cardinality $\left\vert
V\left(  G\right)  \right\vert =n\left(  G\right)  $, and edge set $E=E(G)$ of
size $\left\vert E\left(  G\right)  \right\vert =m\left(  G\right)  $.

If $X\subset V$, then $G[X]$ is the subgraph of $G$ induced by $X$. By $G-v$
we mean the subgraph $G[V-\left\{  v\right\}  ]$, for $v\in V$. The
\textit{neighborhood} of a vertex $v\in V$ is the set $N(v)=\{w:w\in V$ and
$vw\in E\}$. The \textit{neighborhood} of $A\subseteq V$ is $N(A)=\{v\in
V:N(v)\cap A\neq\emptyset\}$, and $N\left[  A\right]  =A\cup N(A)$, or
$N_{G}(A)$ and $N_{G}\left[  A\right]  $, if we specify the graph. If
$A,B\subset V$ are disjoint, then $\left(  A,B\right)  =\left\{  ab:ab\in
E,a\in A,b\in B\right\}  $.

A set $S\subseteq V(G)$ is \textit{independent} if no two vertices from $S$
are adjacent, and by $\mathrm{Ind}(G)$ we mean the family of all the
independent sets of $G$. An independent set of maximum size is a
\textit{maximum independent set} of $G$, and $\alpha(G)=\max\{\left\vert
S\right\vert :S\in\mathrm{Ind}(G)\}$. Let $\Omega(G)$ be the family of all
maximum independent sets, $\mathrm{corona}(G)=%
%TCIMACRO{\tbigcup }%
%BeginExpansion
{\textstyle\bigcup}
%EndExpansion
\{S:S\in\Omega(G)\}$, and $\mathrm{core}(G)=%
%TCIMACRO{\tbigcap }%
%BeginExpansion
{\textstyle\bigcap}
%EndExpansion
\{S:S\in\Omega(G)\}$, while $\xi(G)=\left\vert \mathrm{core}(G)\right\vert $
\cite{LevMan2002a}. A vertex $v\in V(G)$ is $\alpha$-\textit{critical
}provided $\alpha(G-v)<\alpha(G)$. Clearly, $\mathrm{core}(G)$ is the set of
all $\alpha$-critical vertices of $G$. An edge $e\in E(G)$ is $\alpha
$-\textit{critical }provided $\alpha(G)<\alpha(G-e)$. Notice that there are
graphs in which every edge is $\alpha$-critical (e.g., all $C_{2k+1}$ for
$k\geq1$) or no edge is $\alpha$-critical (e.g., all $C_{2k}$ for $k\geq2$).

A \textit{matching} in a graph $G=(V,E)$ is a set of edges $M\subseteq E$ such
that no two edges of $M$ share a common vertex. A matching of maximum
cardinality $\mu(G)$ is a \textit{maximum matching}, and a \textit{perfect
matching} is one saturating all vertices of $G$. Given a matching $M$ in $G$,
a vertex $v\in V$ is called $M$-saturated if there exists an edge $e\in M$
incident with $v$.

An edge $e\in E(G)$ is $\mu$-\textit{critical }provided $\mu(G-e)<\mu(G)$. A
vertex $v\in V(G)$ is $\mu$-\textit{critical (essential) }provided
$\mu(G-v)<\mu(G)$, i.e., $v$ is $M$-saturated by every maximum matching $M$ of
$G$.

It is known that the inequalities $\left\lfloor \frac{n\left(  G\right)  }%
{2}\right\rfloor +1\leq\alpha(G)+\mu(G)\leq n\left(  G\right)  \leq
\alpha(G)+2\mu(G)$ hold for every graph $G$ \cite{BGL2002}. If $\alpha
(G)+\mu(G)=n\left(  G\right)  $, then $G$ is called a K\"{o}nig-Egerv\'{a}ry
graph\textit{ }\cite{dem,ster}. If $S$ is an independent set of a graph $G$
and $A=V\left(  G\right)  -S$, then we write $G=S\ast A$. For instance, if
$E(G[A])=\emptyset$, then $G=S\ast A$ is bipartite; if $G[A]$ is a complete
graph, then $G=S\ast A$ is a split graph.

\begin{theorem}
\label{th715}\cite{LevMan2002a} For a graph $G$, the following properties are equivalent:

\emph{(i)} $G$ is a \textit{K\"{o}nig-Egerv\'{a}ry graph};

\emph{(ii) }$G=S\ast A$, where $S\in$ $\mathrm{Ind}(G)$, $\left\vert
S\right\vert \geq\left\vert A\right\vert $, and $\left(  S,A\right)  $
contains a matching $M$ with $\left\vert M\right\vert =\left\vert A\right\vert
$.
\end{theorem}

\begin{theorem}
\label{th43} If $G$ is a K\"{o}nig-Egerv\'{a}ry graph, then

\emph{(i)} \cite{LevMan2003} $\mathrm{corona}(G)\cup N_{G}\left(
\mathrm{core}(G)\right)  =V(G);$

\emph{(ii) }\cite{LevMan2012a} $d(G)=\alpha(G)-\mu(G)=\left\vert
\mathrm{core}(G)\right\vert -\left\vert N(\mathrm{core}(G))\right\vert ;$

\emph{(iii)} \cite{LevMan2014} $\left\vert \mathrm{core}\left(  G\right)
\right\vert +\left\vert \mathrm{corona}\left(  G\right)  \right\vert
=2\alpha(G).$
\end{theorem}

For $X\subseteq V(G)$, the number $\left\vert X\right\vert -\left\vert
N(X)\right\vert $ is the \textit{difference} of $X$, denoted $d(X)$. The
\textit{critical difference} $d(G)$ is $\max\{d(I):I\in\mathrm{Ind}(G)\}$. If
$A\in\mathrm{Ind}(G)$ with $d\left(  X\right)  =d(G)$, then $A$ is a
\textit{critical independent set} \cite{Zhang1990}. For a graph $G$, let
$\mathrm{MaxCritIndep}(G)=\{S:S$ \textit{is a maximum critical independent
set}$\}$ \cite{LevMan2022c}, $\mathrm{\ker}(G)$ be the intersection of all its
critical independent sets and $\varepsilon(G)=\left\vert \mathrm{\ker
}(G)\right\vert $ \cite{LevMan2012a,LevMan2012c}. The \textit{critical
independence number} of a graph $G$, denoted as $\alpha^{\prime}\left(
G\right)  $, is the cardinality of a maximum critical independent set
\cite{Larson2011}.

\begin{theorem}
\label{th333} \emph{(i)} \cite{ButTruk2007} Each critical independent set is
included in some $S\in\Omega(G)$.

\emph{(ii)} \cite{Larson2007} Every critical independent set is contained in
some $S\in\mathrm{MaxCritIndep}(G)$.

\emph{(iii)} \cite{Larson2007} There is a matching from $N(S)$ into $S$ for
every critical independent set $S$.
\end{theorem}

Some recent developments involving critical independent sets may be found in
\cite{HongliangYang2023,YHQ2022}.

\begin{theorem}
\label{th100}\cite{Larson2011} For any graph $G$, there is a unique set
$X\subseteq V(G)$ such that

\emph{(i)} $X=N\left[  A\right]  $\ for every $A\in\mathrm{MaxCritIndep}(G)$; \ 

\emph{(ii)} $G\left[  X\right]  $ is a \textit{K\"{o}nig-Egerv\'{a}ry} graph.
\end{theorem}

Let $\mathrm{nucleus}(G)=\bigcap\{S:S\in\mathrm{MaxCritIndep}(G)\}$,
$\mathrm{diadem}(G)=\bigcup\{S:S$ is a critical independent set in $G\}$
\cite{JarLevMan2019}, while $\beta(G)=\left\vert \mathrm{diadem}(G)\right\vert
$ \cite{LevMan2024a}.

If $\alpha(G)+\mu(G)=n\left(  G\right)  -1$, then $G$ is a $1$%
-\textit{K\"{o}nig-Egerv\'{a}ry graph }\cite{LevMan2024a}. The
\textit{K\"{o}nig deficiency} of graph $G$ is $\kappa\left(  G\right)
=n\left(  G\right)  -\left(  \alpha(G)+\mu(G)\right)  $ \cite{BartaKres2020}.
Thus, a graph $G$ is $1$-K\"{o}nig-Egerv\'{a}ry if and only if $\kappa\left(
G\right)  =1$.

A graph is: \textit{(i)} \textit{unicyclic} if it has a unique cycle
\cite{LevMan2012b}, and \textit{(ii)} \textit{almost bipartite} if it has only
one odd cycle \cite{LevMan2022}.

\begin{lemma}
\label{lem84}\cite{LevMan2022} If $G$ is an almost bipartite graph, then
$n(G)-1\leq\alpha(G)+\mu(G)\leq n(G)$.
\end{lemma}

Consequently, one may say that each almost bipartite graph is either a
K\"{o}nig-Egerv\'{a}ry graph or a $1$-K\"{o}nig-Egerv\'{a}ry graph. For
instance, both $G_{1}$ and $G_{2}$ from Figure \ref{fig123} are $1$%
-K\"{o}nig-Egerv\'{a}ry graphs, while only $G_{1}$ is almost bipartite.

\begin{figure}[h]
\setlength{\unitlength}{1cm}\begin{picture}(5,1.2)\thicklines
\multiput(2,0)(1,0){4}{\circle*{0.29}}
\multiput(4,1)(1,0){2}{\circle*{0.29}}
\put(2,1){\circle*{0.29}}
\put(2,0){\line(1,0){3}}
\put(2,0){\line(0,1){1}}
\put(2,1){\line(1,-1){1}}
\put(4,0){\line(0,1){1}}
\put(4,1){\line(1,0){1}}
\put(5,0){\line(0,1){1}}
\put(1,0.5){\makebox(0,0){$G_{1}$}}
\multiput(8,0)(1,0){5}{\circle*{0.29}}
\multiput(11,1)(1,0){2}{\circle*{0.29}}
\put(8,1){\circle*{0.29}}
\put(8,0){\line(1,0){4}}
\put(8,0){\line(0,1){1}}
\put(8,1){\line(1,-1){1}}
\put(11,1){\line(1,0){1}}
\put(10,0){\line(1,1){1}}
\put(12,0){\line(0,1){1}}
\put(7,0.5){\makebox(0,0){$G_{2}$}}
\end{picture}\caption{$1$-K\"{o}nig-Egerv\'{a}ry graphs}%
\label{fig123}%
\end{figure}
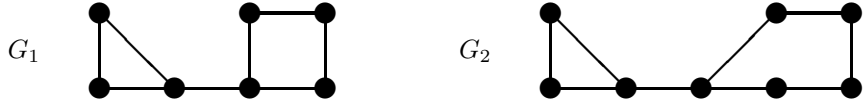

\begin{theorem}
\label{th44} If $G$ is an almost bipartite non-K\"{o}nig-Egerv\'{a}ry graph, then

\emph{(i) }\cite{LevMan2023a} $\mathrm{corona}\left(  G\right)  \cup
N_{G}\left(  \mathrm{core}\left(  G\right)  \right)  =V\left(  G\right)  ;$

\emph{(ii) }\cite{LevMan2022} $d(G)=\alpha(G)-\mu(G)=\left\vert \mathrm{core}%
(G)\right\vert -\left\vert N(\mathrm{core}(G))\right\vert ;$

\emph{(iii) }\cite{LevMan2023a} $\left\vert \mathrm{core}\left(  G\right)
\right\vert +\left\vert \mathrm{corona}\left(  G\right)  \right\vert
=2\alpha\left(  G\right)  +1$.
\end{theorem}

Lemma \ref{lem84} and Theorem \ref{th44}(\textit{ii}) directly imply the following.

\begin{corollary}
\label{cor8}If $G$ is an almost bipartite non-K\"{o}nig-Egerv\'{a}ry graph,
then%
\[
n\left(  G\right)  +d\left(  G\right)  =2\alpha(G)+1,\alpha(G)=\frac{n\left(
G\right)  +d\left(  G\right)  -1}{2},\mu(G)=\frac{n\left(  G\right)  -d\left(
G\right)  -1}{2},
\]
and consequently, $\mu(G)<\frac{n\left(  G\right)  }{2}$, i.e., $G$ has no
perfect matchings.
\end{corollary}

It is worth mentioning that the order of an almost bipartite
non-K\"{o}nig-Egerv\'{a}ry graph may be even (see the graph $G_{2}$ in Figure
\ref{fig121212}).

\begin{theorem}
\label{th444}\cite{LevMan2012a} For a graph $G$, the following assertions are true:

\emph{(i)} $\mathrm{\ker}(G)\subseteq\mathrm{core}(G)$;

\emph{(ii)} if $A$ and $B$ are critical in $G$, then $A\cup B$ and $A\cap B$
are critical as well;

\emph{(iii)} $G$ has a unique minimal independent critical set, namely,
$\mathrm{\ker}(G)$.
\end{theorem}

\begin{theorem}
\label{th2222}\cite{LevMan2013c,LevMan2022b} If $G$ is bipartite, or an almost
bipartite non-K\"{o}nig-Egerv\'{a}ry graph, then $\mathrm{\ker}%
(G)=\mathrm{core}(G)$.
\end{theorem}

Let $\varrho_{v}\left(  G\right)  $ denote the number of vertices $v\in
V\left(  G\right)  $, such that $G-v$ is a K\"{o}nig-Egerv\'{a}ry graph, and
$\varrho_{e}\left(  G\right)  $ denote the number of edges $e\in E\left(
G\right)  $ satisfying $G-e$ is a K\"{o}nig-Egerv\'{a}ry graph
\cite{LevMan2024}.

\begin{theorem}
\label{th9}\cite{LevMan2024} If $G$ is a K\"{o}nig-Egerv\'{a}ry graph, then%
\[
\varrho_{v}\left(  G\right)  =n\left(  G\right)  -\xi\left(  G\right)
+\varepsilon\left(  G\right)  \text{ and }\varrho_{e}\left(  G\right)  \leq
m\left(  G\right)  -\xi\left(  G\right)  +\varepsilon\left(  G\right)  .
\]

\end{theorem}

\begin{theorem}
\cite{LevMan2024a}\label{th10} If $G$ is a $1$-K\"{o}nig-Egerv\'{a}ry graph,
then%
\[
\varrho_{v}\left(  G\right)  \leq n\left(  G\right)  +d\left(  G\right)
-\xi\left(  G\right)  -\beta\left(  G\right)  =n\left(  G\right)  -\xi\left(
G\right)  +\left\vert N\left(  \mathrm{diadem}(G)\right)  \right\vert .
\]

\end{theorem}

In this paper, we show that if $G$ is an almost bipartite
non-K\"{o}nig-Egerv\'{a}ry graph with the unique odd cycle $C$, then $\left\{
V(C),N_{G}\left[  \mathrm{diadem}\left(  G\right)  \right]  \right\}  $ forms
a partition for $V\left(  G\right)  $, and%
\begin{gather*}
\varrho_{v}\left(  G\right)  =n\left(  G\right)  +d\left(  G\right)
-\xi\left(  G\right)  -\beta(G)=\\
\left\vert V\left(  C\right)  \right\vert +\left\vert \mathrm{nucleus}\left(
G\right)  \right\vert -\left\vert \mathrm{core}(G)\right\vert =\left\vert
\mathrm{corona}(G)\right\vert -\left\vert \mathrm{diadem}(G)\right\vert .
\end{gather*}

\section{The $\left\{  V(C),N_{G}\left[  \emph{diadem}\left(  G\right)
\right]  \right\}  $ partition of $V(G)$}

\begin{lemma}
\label{lem17}For every graph $G$ and each pair $A,B\subseteq V\left(
G\right)  $ the following equality holds:
\[%
%TCIMACRO{\dsum \limits_{a\in A}}%
%BeginExpansion
{\displaystyle\sum\limits_{a\in A}}
%EndExpansion
\left\vert N_{G}\left(  a\right)  \cap B\right\vert =%
%TCIMACRO{\dsum \limits_{b\in B}}%
%BeginExpansion
{\displaystyle\sum\limits_{b\in B}}
%EndExpansion
\left\vert N_{G}\left(  b\right)  \cap A\right\vert .
\]

\end{lemma}

\begin{proof}
The proof is based on the double counting method. Namely, every edge outgoing
from either $A-B$ or $B-A$ is counted just once, while every edge outgoing
from $A\cap B$ is counted twice in the both sides of the equality.
\end{proof}

\begin{corollary}
\label{cor11}For every graph $G$ and each pair $A,B\subseteq V\left(
G\right)  $, the set $N_{G}\left(  A\right)  \cap B=\emptyset$ if and only if
the set $N_{G}\left(  B\right)  \cap A=\emptyset$.
\end{corollary}

\begin{proof}
Clearly, $N_{G}\left(  A\right)  \cap B=\emptyset$ if and only if
$N_{G}\left(  a\right)  \cap B=\emptyset$ for every $a\in A$, i.e., $%
%TCIMACRO{\dsum \limits_{a\in A}}%
%BeginExpansion
{\displaystyle\sum\limits_{a\in A}}
%EndExpansion
\left\vert N_{G}\left(  a\right)  \cap B\right\vert =0$. Hence, by Lemma
\ref{lem17}, if $N_{G}\left(  A\right)  \cap B=\emptyset$, then
\[%
%TCIMACRO{\dsum \limits_{b\in B}}%
%BeginExpansion
{\displaystyle\sum\limits_{b\in B}}
%EndExpansion
\left\vert N_{G}\left(  b\right)  \cap A\right\vert =0\text{, i.e., }%
N_{G}\left(  B\right)  \cap A=\emptyset.
\]
Now, the reasoning based on the symmetry between $A$ and $B$ completes the proof.
\end{proof}

\begin{proposition}
\label{prop14}For every graph $G$, the following equality holds:%
\[
\left\vert N_{G}\left[  \mathrm{diadem}\left(  G\right)  \right]  \right\vert
=\beta(G)+\left\vert \mathrm{nucleus}\left(  G\right)  \right\vert -d\left(
G\right)  .
\]

\end{proposition}

\begin{proof}
If $v\in N_{G}\left[  \mathrm{diadem}\left(  G\right)  \right]  $, then either
there exists a maximum critical independent set $A$ such that $v\notin A$ or
$v\in%
%TCIMACRO{\dbigcap \limits_{A\in MaxCrit(G)}}%
%BeginExpansion
{\displaystyle\bigcap\limits_{A\in MaxCrit(G)}}
%EndExpansion
A=\mathrm{nucleus}\left(  G\right)  $. 

If $v\notin A$, then $v\in N_{G}\left(  A\right)  $, since $N_{G}\left[
\mathrm{diadem}\left(  G\right)  \right]  =A\cup N_{G}\left(  A\right)  $ by
Theorem \ref{th100}\emph{(i)}. In other words, if $v\in N_{G}\left[
\mathrm{diadem}\left(  G\right)  \right]  $, then either $v\in N_{G}\left(
\mathrm{diadem}\left(  G\right)  \right)  $ or $v\in\mathrm{nucleus}\left(
G\right)  $, i.e.,
\[
N_{G}\left[  \mathrm{diadem}\left(  G\right)  \right]  =N_{G}\left(
\mathrm{diadem}\left(  G\right)  \right)  \cup\mathrm{nucleus}\left(
G\right)  .
\]
Since all $A\in MaxCrit(G)$ are independent, we know that
\[
N_{G}\left(  \mathrm{diadem}\left(  G\right)  \right)  \cap\mathrm{nucleus}%
\left(  G\right)  =\emptyset.
\]
Thus,
\begin{gather*}
\left\vert N_{G}\left[  \mathrm{diadem}\left(  G\right)  \right]  \right\vert
=\left\vert N_{G}\left(  \mathrm{diadem}\left(  G\right)  \right)  \right\vert
+\left\vert \mathrm{nucleus}\left(  G\right)  \right\vert =\\
\left\vert \mathrm{diadem}\left(  G\right)  \right\vert -d\left(  G\right)
+\left\vert \mathrm{nucleus}\left(  G\right)  \right\vert =\\
\beta(G)-d\left(  G\right)  +\left\vert \mathrm{nucleus}\left(  G\right)
\right\vert ,
\end{gather*}
since, according to its definition and Theorem \ref{th444}\emph{(ii)},
$\mathrm{diadem}\left(  G\right)  $ is critical in $G$.
\end{proof}

\begin{proposition}
\cite{LevMan2019} \label{prop13}$\left\vert \mathrm{diadem}\left(  G\right)
\right\vert +\left\vert \mathrm{nucleus}\left(  G\right)  \right\vert
\leq2\alpha^{\prime}\left(  G\right)  $ for every graph $G$.
\end{proposition}

Proposition \ref{prop14} and Proposition \ref{prop13} immediately imply the following.

\begin{corollary}
$\left\vert N_{G}\left[  \mathrm{diadem}\left(  G\right)  \right]  \right\vert
\leq2\alpha^{\prime}\left(  G\right)  -d\left(  G\right)  $ for every graph
$G$.
\end{corollary}

Consider the graphs in Figure \ref{fig121212}: $\mathrm{core}\left(
G_{1}\right)  =\left\{  a,b\right\}  $ and $c\in N\left(  \mathrm{core}\left(
G_{1}\right)  \right)  $ is $\mu$-critical, $\mathrm{core}\left(
G_{2}\right)  =\emptyset$ and $x$ is $\mu$-critical\textit{.}\begin{figure}[h]
\setlength{\unitlength}{1cm}\begin{picture}(5,1.2)\thicklines
\multiput(3,0)(1,0){4}{\circle*{0.29}}
\multiput(3,1)(1,0){4}{\circle*{0.29}}
\put(3,0){\line(1,0){3}}
\put(3,0){\line(0,1){1}}
\put(3,1){\line(1,0){1}}
\put(4,1){\line(1,-1){1}}
\put(5,1){\line(1,-1){1}}
\put(6,0){\line(0,1){1}}
\put(4.65,1){\makebox(0,0){$a$}}
\put(6.3,1){\makebox(0,0){$b$}}
\put(6.3,0){\makebox(0,0){$c$}}
\put(2.2,0.5){\makebox(0,0){$G_{1}$}}
\multiput(8,0)(1,0){4}{\circle*{0.29}}
\multiput(9,1)(1,0){3}{\circle*{0.29}}
\put(8,0){\line(1,0){3}}
\put(8,0){\line(1,1){1}}
\put(9,1){\line(1,0){1}}
\put(10,0){\line(0,1){1}}
\put(11,0){\line(0,1){1}}
\put(10.65,0.3){\makebox(0,0){$x$}}
\put(7.2,0.5){\makebox(0,0){$G_{2}$}}
\end{picture}\caption{Almost bipartite non-K\"{o}nig-Egerv\'{a}ry graphs.}%
\label{fig121212}%
\end{figure}
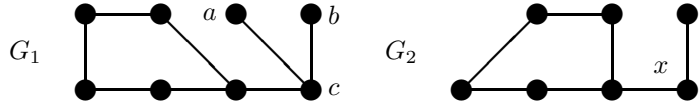

Notice that, by Lemma \ref{lem84}, a unicyclic graph $G$ may be either
K\"{o}nig-Egerv\'{a}ry or $1$-K\"{o}nig-Egerv\'{a}ry. Clearly, $\left\vert
V\left(  C\right)  \right\vert \leq\varrho_{v}\left(  G\right)  \leq n\left(
G\right)  $, where $C$ is the unique cycle of $G$. If $G$ is a
K\"{o}nig-Egerv\'{a}ry graph, then $\varrho_{v}\left(  G\right)  =n\left(
G\right)  -\xi\left(  G\right)  +\varepsilon\left(  G\right)  $, according to
Theorem \ref{th9}\emph{(iv)}. This formula is true for some unicyclic
non-K\"{o}nig-Egerv\'{a}ry graphs (e.g., $G=C_{2k+1}$). On the other hand, the
formula is not true for the graphs $G_{1}$ and $G_{2}$ from Figure
\ref{fig121212}. They are unicyclic non-K\"{o}nig-Egerv\'{a}ry graphs and have
$\varrho_{v}\left(  G_{1}\right)  =5=\left\vert V\left(  G_{1}\right)
-\left\{  a,b,c\right\}  \right\vert $ and $\varrho_{v}\left(  G_{2}\right)
=6=\left\vert V\left(  G_{2}\right)  -\left\{  x\right\}  \right\vert $. In
general, if $G$ is a unicyclic \ non-K\"{o}nig-Egerv\'{a}ry graph, then
$\xi\left(  G\right)  =\varepsilon\left(  G\right)  $ by Theorem \ref{th2222},
i.e., $\varrho_{v}\left(  G\right)  =n\left(  G\right)  $.

\begin{lemma}
\cite{LevMan2022}\label{lem7} Let $G$ be an almost bipartite
non-K\"{o}nig-Egerv\'{a}ry graph with the unique odd cycle $C$. Then the
following assertions are true:

\emph{(i)} if $A$ is a critical independent set, then $A\cap V(C)=$
$\emptyset$;

\emph{(ii)} $\mathrm{core}\left(  G\right)  $ is a critical set.
\end{lemma}

\begin{theorem}
\cite{LevMan2024a}\label{th11} Suppose that $A\in\mathrm{Ind}(G)$. If there is
a matching from $N_{G}(A)$ into $A$, then\textit{ every matching }from
$N_{G}(A)$ into $A$ can be enlarged to a maximum matching of $G$, and every
vertex of $N_{G}(A)$ is $\mu$-\textit{critical.}
\end{theorem}

\begin{theorem}
\cite{LevMan2024a}\label{th17} Let $G$ be a $1$-K\"{o}nig-Egerv\'{a}ry graph.
Then $G-v$ is K\"{o}nig-Egerv\'{a}ry if and only if the vertex $v$ is neither
$\alpha$-\textit{critical nor }$\mu$-\textit{critical.}
\end{theorem}

\begin{proposition}
\label{prop15}If $G$ is an almost bipartite non-K\"{o}nig-Egerv\'{a}ry graph
with the unique odd cycle $C$, then
\[
V(C)\cap N_{G}\left[  \mathrm{diadem}\left(  G\right)  \right]  =\emptyset
\text{ or, equivalently, }N_{G}\left[  V(C)\right]  \cap\mathrm{diadem}\left(
G\right)  =\emptyset\mathbf{.}%
\]

\end{proposition}

\begin{proof}
Let $A$ be a critical independent set. Theorem \ref{th11} claims that every
vertex of $N_{G}(A)$ is $\mu$-critical.\textit{ }On the other hand, $G-v$ is
bipartite for every $v\in V\left(  C\right)  $, since $C$ is the only odd
cycle in $G$. Hence, by Lemma \ref{lem84} and Theorem \ref{th17} \ no vertex
from $V\left(  C\right)  $ is $\mu$-critical. Thus, $V(C)\cap N_{G}\left(
A\right)  =\emptyset$. On the other hand, by Lemma \ref{lem7}, we know that
$V(C)\cap A=$ $\emptyset$. Thus, $V(C)\cap N_{G}\left[  A\right]  =\emptyset$.
Finally,%
\begin{align*}
V(C)\cap N_{G}\left[  \mathrm{diadem}\left(  G\right)  \right]   &  =V(C)\cap%
%TCIMACRO{\dbigcup \limits_{A\in Crit(G)}}%
%BeginExpansion
{\displaystyle\bigcup\limits_{A\in Crit(G)}}
%EndExpansion
N_{G}\left[  A\right]  =\\%
%TCIMACRO{\dbigcup \limits_{A\in Crit(G)}}%
%BeginExpansion
{\displaystyle\bigcup\limits_{A\in Crit(G)}}
%EndExpansion
\left(  V(C)\cap N_{G}\left[  A\right]  \right)   &  =\emptyset,
\end{align*}
as required. Hence, $N_{G}\left[  V(C)\right]  \cap\mathrm{diadem}\left(
G\right)  =\emptyset$, by Corollary \ref{cor11}.
\end{proof}

For an almost bipartite graph $G$ with $C$ as its unique odd cycle, let
\[
N_{1}(C)=\{v:v\in V\left(  G\right)  -V(C),N(v)\cap V(C)\neq\emptyset\},
\]
and $D_{y}=(V_{y},E_{y})$ be the connected bipartite subgraph of $G-E(C)$
containing $y$, for every $y\in V(C)$ (see Figure \ref{fig44} for an example).
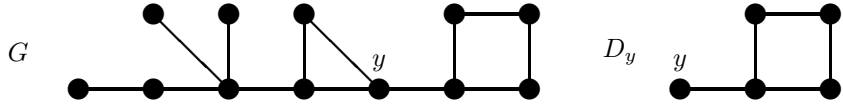
\begin{figure}[h]
\setlength{\unitlength}{1cm}\begin{picture}(5,1.2)\thicklines
\multiput(2,0)(1,0){7}{\circle*{0.29}}
\multiput(3,1)(1,0){3}{\circle*{0.29}}
\multiput(7,1)(1,0){2}{\circle*{0.29}}
\put(2,0){\line(1,0){6}}
\put(3,1){\line(1,-1){1}}
\put(4,0){\line(0,1){1}}
\put(5,1){\line(1,-1){1}}
\put(5,0){\line(0,1){1}}
\put(7,0){\line(0,1){1}}
\put(7,1){\line(1,0){1}}
\put(8,0){\line(0,1){1}}
\put(6,0.35){\makebox(0,0){$y$}}
\put(1.2,0.5){\makebox(0,0){$G$}}
\multiput(10,0)(1,0){3}{\circle*{0.29}}
\multiput(11,1)(1,0){2}{\circle*{0.29}}
\put(10,0){\line(1,0){2}}
\put(11,0){\line(0,1){1}}
\put(11,1){\line(1,0){1}}
\put(12,0){\line(0,1){1}}
\put(10,0.35){\makebox(0,0){$y$}}
\put(9.2,0.5){\makebox(0,0){$D_{y}$}}
\end{picture}\caption{An almost bipartite non-K\"{o}nig-Egerv\'{a}ry graph and
its $D_{y}$.}%
\label{fig44}%
\end{figure}

\begin{lemma}
\label{lem10}If $G$ is an almost bipartite non-K\"{o}nig-Egerv\'{a}ry graph
with the unique odd cycle $C$, then $N_{G}\left[  V\left(  D_{y}-y\right)
\right]  \cap N_{G}\left[  V\left(  D_{x}-x\right)  \right]  =\emptyset$ for
every pair of different vertices $x,y\in V(C)$.
\end{lemma}

\begin{proof}
Suppose, to the contrary, that there is some vertex $t\in N_{G}\left[
V\left(  D_{y}-y\right)  \right]  \cap N_{G}\left[  V\left(  D_{x}-x\right)
\right]  $. It means that there exists a cordless path $P$\ connecting $x$ and
$y$, such that $V\left(  P\right)  \subseteq V\left(  D_{y}-y\right)  \cup
V\left(  D_{x}-x\right)  $, in addition to two canonical paths $P_{1}$,
$P_{2}$ joining $x$ and $y$ on $C$. Since $C$ is an odd cycle, then either $P$
and $P_{1}$ or $P$ and $P_{2}$ comprise an extra odd cycle.
\end{proof}

\begin{theorem}
\label{conj1}If $G$ is an almost bipartite non-K\"{o}nig-Egerv\'{a}ry graph
with the unique odd cycle $C$, then
\[
V(C)\cup N_{G}\left[  \mathrm{diadem}\left(  G\right)  \right]  =V\left(
G\right)  .
\]

\end{theorem}

\begin{proof}
Suppose, to the contrary, that $V\left(  G\right)  -\left(  V(C)\cup
N_{G}\left[  \mathrm{diadem}\left(  G\right)  \right]  \right)  =A\neq
\emptyset$. Since $C$ is the unique odd cycle in $G$, the graph $G\left[
A\right]  $ is bipartite.

Let $y\in V\left(  C\right)  $ and $A_{y}=A\cap V\left(  D_{y}-y\right)
-N_{G}\left(  y\right)  $.

If $A_{y}\neq\emptyset$, then there is a non-empty maximum independent set of
$G\left[  A_{y}\right]  $, say $B$. Clearly, the graph $G\left[  A_{y}\right]
$ is bipartite, as a subgraph of $G\left[  A\right]  $. Hence,
\[
\left\vert B\right\vert =\alpha\left(  G\left[  A_{y}\right]  \right)  \geq
\mu\left(  G\left[  A_{y}\right]  \right)  =\left\vert N_{G\left[
A_{y}\right]  }\left(  B\right)  \right\vert ,
\]
because $\alpha\left(  H\right)  \geq\mu\left(  H\right)  $ holds for every
K\"{o}nig-Egerv\'{a}ry graph $H$, by Theorem \ref{th715}\emph{(ii)}. By
Theorem \ref{th100}\emph{(i)}, $N_{G}\left[  \mathrm{diadem}\left(  G\right)
\right]  =N_{G}\left[  X\right]  $ for every $X\in\mathrm{MaxCritIndep}(G)$.
Taking into account that $B\subseteq A_{y}$ and Lemma \ref{lem10}, we obtain
$N_{G}\left(  B\right)  \cap V(C)=\emptyset$. Since $A\cap N_{G}\left(
X\right)  =\emptyset$, by Corollary \ref{cor11}, we obtain $N_{G}\left(
A\right)  \cap X=\emptyset$, and, consequently, $N_{G}\left(  B\right)  \cap
X=\emptyset$, because $B\subseteq A$. Thus,
\[
N_{G}\left(  B\right)  =N_{G\left[  A_{y}\right]  }\left(  B\right)
\cup\left(  N_{G}\left(  B\right)  \cap N_{G}\left(  X\right)  \right)  .
\]
Therefore,
\[
d\left(  G\right)  \geq d_{G}\left(  B\cup X\right)  =\left\vert B\right\vert
-\left\vert N_{G\left[  A_{y}\right]  }\left(  B\right)  \right\vert
+\left\vert X\right\vert -\left\vert N_{G}\left(  X\right)  \right\vert \geq
d_{G}\left(  X\right)  =d\left(  G\right)  ,
\]
which means that the set $B\cup X$ is critical. On the other hand, the set
$B\cup X$ is independent, by the definition of the set $A$, and the fact that
$B$ and $X$ are independent in $G$. It implies that $B=\emptyset$, because
$X\in\mathrm{MaxCritIndep}(G)$. In other words, $\alpha\left(  G\left[
A_{y}\right]  \right)  =0$ and, consequently, $A_{y}=\emptyset$.

Thus, $A\cap V\left(  D_{y}-y\right)  \subseteq N_{G}\left(  y\right)  $.
Hence, the set $A\cap V\left(  D_{y}-y\right)  $ is independent, otherwise,
$G$ has more than one odd cycle. Moreover, by Lemma \ref{lem10}, we have
\[
\left\{  y\right\}  \subseteq N_{G}\left(  A\cap V\left(  D_{y}-y\right)
\right)  \subseteq N_{G}\left(  X\right)  \cup\left\{  y\right\}  .
\]
Therefore,
\begin{gather*}
d\left(  G\right)  \geq d_{G}\left(  \left(  A\cap V\left(  D_{y}-y\right)
\right)  \cup X\right)  =\\
\left\vert A\cap V\left(  D_{y}-y\right)  \right\vert -\left\vert \left\{
y\right\}  \right\vert +\left\vert X\right\vert -\left\vert N_{G}\left(
X\right)  \right\vert \geq d_{G}\left(  X\right)  =d\left(  G\right)  ,
\end{gather*}
which means that the set $\left(  A\cap V\left(  D_{y}-y\right)  \right)  \cup
X$ is critical. On the other hand, the set $\left(  A\cap V\left(
D_{y}-y\right)  \right)  \cup X$ is independent, since $A\cap V\left(
D_{y}-y\right)  $ and $X$ are independent in $G$, and the fact that $A\cap
N_{G}\left[  \mathrm{diadem}\left(  G\right)  \right]  =\emptyset$. It implies
that $A\cap V\left(  D_{y}-y\right)  =\emptyset$, because $X\in
\mathrm{MaxCritIndep}(G)$.

Consequently, $A\cap V\left(  D_{y}-y\right)  =\emptyset$ for every $y\in
V\left(  C\right)  $, which implies
\[
A=%
%TCIMACRO{\dbigcup \limits_{y\in V\left(  C\right)  }}%
%BeginExpansion
{\displaystyle\bigcup\limits_{y\in V\left(  C\right)  }}
%EndExpansion
\left(  A\cap V\left(  D_{y}-y\right)  \right)  =\emptyset.
\]

In conclusion, $V(C)\cup N_{G}\left[  \mathrm{diadem}\left(  G\right)
\right]  =V\left(  G\right)  $, as claimed.
\end{proof}

In summary, Proposition \ref{prop15} and Theorem \ref{conj1} claim that the
sets $V(C)$ and $N_{G}\left[  \mathrm{diadem}\left(  G\right)  \right]  $ form
a partition of $V\left(  G\right)  $, for every almost bipartite
non-K\"{o}nig-Egerv\'{a}ry graph $G$ with the unique odd cycle $C$.

\begin{proposition}
\label{prop17} For every graph $G$, the following equality holds:%
\[
N_{G}\left(  \mathrm{diadem}\left(  G\right)  \right)  -\mathrm{diadem}\left(
G\right)  =%
%TCIMACRO{\dbigcap \limits_{A\in\mathrm{MaxCritIndep}(G)}}%
%BeginExpansion
{\displaystyle\bigcap\limits_{A\in\mathrm{MaxCritIndep}(G)}}
%EndExpansion
N_{G}\left(  A\right)  .
\]

\end{proposition}

\begin{proof}
Let $x\in%
%TCIMACRO{\dbigcap \limits_{A\in\mathrm{MaxCritIndep}(G)}}%
%BeginExpansion
{\displaystyle\bigcap\limits_{A\in\mathrm{MaxCritIndep}(G)}}
%EndExpansion
N_{G}\left(  A\right)  $. It follows that $x\in N_{G}\left(  A\right)  $ for
every $A\in\mathrm{MaxCritIndep}(G)$, i.e., $x$ has a neighbor in every
$A\in\mathrm{MaxCritIndep}(G)$. Therefore,
\[
x\notin%
%TCIMACRO{\dbigcup \limits_{A\in\mathrm{MaxCritIndep}(G)}}%
%BeginExpansion
{\displaystyle\bigcup\limits_{A\in\mathrm{MaxCritIndep}(G)}}
%EndExpansion
A=\mathrm{diadem}\left(  G\right)  ,
\]
since each $A$ is independent. Hence,
\[
x\in%
%TCIMACRO{\dbigcup \limits_{A\in\mathrm{MaxCritIndep}(G)}}%
%BeginExpansion
{\displaystyle\bigcup\limits_{A\in\mathrm{MaxCritIndep}(G)}}
%EndExpansion
N_{G}\left(  A\right)  =N_{G}\left(  \mathrm{diadem}\left(  G\right)  \right)
.
\]
Consequently, we get
\[%
%TCIMACRO{\dbigcap \limits_{A\in\mathrm{MaxCritIndep}(G)}}%
%BeginExpansion
{\displaystyle\bigcap\limits_{A\in\mathrm{MaxCritIndep}(G)}}
%EndExpansion
N_{G}\left(  A\right)  \subseteq N_{G}\left(  \mathrm{diadem}\left(  G\right)
\right)  -\mathrm{diadem}\left(  G\right)  .
\]

Conversely, let $x\in N_{G}\left(  \mathrm{diadem}\left(  G\right)  \right)
-\mathrm{diadem}\left(  G\right)  $. Thus, $x\notin A$, for every
$A\in\mathrm{MaxCritIndep}(G)$, because
\[
\mathrm{diadem}\left(  G\right)  =%
%TCIMACRO{\dbigcup \limits_{A\in\mathrm{MaxCritIndep}(G)}}%
%BeginExpansion
{\displaystyle\bigcup\limits_{A\in\mathrm{MaxCritIndep}(G)}}
%EndExpansion
A.
\]

Assume, to the contrary, that there exists some $A_{x}\in\mathrm{MaxCritIndep}%
(G)$, such that $x\notin N_{G}\left(  A_{x}\right)  $. On the other hand,
there is $A^{x}\in\mathrm{MaxCritIndep}(G)$, such that $x\in N_{G}\left(
A^{x}\right)  $, because $x\in N_{G}\left(  \mathrm{diadem}\left(  G\right)
\right)  $. By Theorem \ref{th100}\emph{(i)}, the set $A_{x}\cup N_{G}\left(
A_{x}\right)  $ coincides with the set $A^{x}\cup N_{G}\left(  A^{x}\right)
$, which is impossible, since we know that $x\in A^{x}\cup N_{G}\left(
A^{x}\right)  $ and $x\notin A_{x}\cup N_{G}\left(  A_{x}\right)  $.
\end{proof}

\begin{lemma}
\label{lem13}If $G$ is an almost bipartite non-K\"{o}nig-Egerv\'{a}ry graph
with the unique cycle $C$, then $\left(  N_{G}\left(  x\right)  \cap
N_{G}\left(  y\right)  \right)  -V\left(  C\right)  =\emptyset$ for every
distinct $x,y\in V\left(  C\right)  $, and consequently,%
\[
N_{G}\left(  V\left(  C\right)  \right)  -V\left(  C\right)  \subseteq
N_{G}\left(  \mathrm{diadem}\left(  G\right)  \right)  -\mathrm{diadem}\left(
G\right)
\]
and
\[%
%TCIMACRO{\dsum \limits_{v\in V\left(  C\right)  }}%
%BeginExpansion
{\displaystyle\sum\limits_{v\in V\left(  C\right)  }}
%EndExpansion
\left(  \deg\left(  v\right)  -2\right)  \leq\left\vert N_{G}\left(
\mathrm{diadem}\left(  G\right)  \right)  -\mathrm{diadem}\left(  G\right)
\right\vert .
\]

\end{lemma}

\begin{proof}
Assume, to the contrary, that there exists some $z\in\left(  N_{G}\left(
x\right)  \cap N_{G}\left(  y\right)  \right)  -V\left(  C\right)  $. Then the
edges $xz,zy$ together with one of the paths on $C$ connecting $x$ and $y$,
bring birth to an additional odd cycle in $G$, in contradiction with the
hypothesis on $G$.

Since $\left(  N_{G}\left(  x\right)  \cap N_{G}\left(  y\right)  \right)
-V\left(  C\right)  =\emptyset$ for every distinct $x,y\in V\left(  C\right)
$, we obtain%
\[%
%TCIMACRO{\dsum \limits_{v\in V\left(  C\right)  }}%
%BeginExpansion
{\displaystyle\sum\limits_{v\in V\left(  C\right)  }}
%EndExpansion
\left(  \deg\left(  v\right)  -2\right)  =\left\vert N_{G}\left(  V\left(
C\right)  \right)  -V\left(  C\right)  \right\vert .
\]

By Theorem \ref{th100}\emph{(i}), Proposition\emph{ }\ref{prop15} and Theorem
\ref{conj1}, $N_{G}\left(  V\left(  C\right)  \right)  -V\left(  C\right)
\subseteq N_{G}\left(  A\right)  $ for every $A\in\mathrm{MaxCritIndep}(G)$.

Hence,
\[
N_{G}\left(  V\left(  C\right)  \right)  -V\left(  C\right)  \subseteq%
%TCIMACRO{\dbigcap \limits_{A\in\mathrm{MaxCritIndep}(G)}}%
%BeginExpansion
{\displaystyle\bigcap\limits_{A\in\mathrm{MaxCritIndep}(G)}}
%EndExpansion
N_{G}\left(  A\right)  .
\]
Therefore,
\[%
%TCIMACRO{\dsum \limits_{v\in V\left(  C\right)  }}%
%BeginExpansion
{\displaystyle\sum\limits_{v\in V\left(  C\right)  }}
%EndExpansion
\left(  \deg\left(  v\right)  -2\right)  \leq\left\vert
%TCIMACRO{\dbigcap \limits_{A\in\mathrm{MaxCritIndep}(G)}}%
%BeginExpansion
{\displaystyle\bigcap\limits_{A\in\mathrm{MaxCritIndep}(G)}}
%EndExpansion
N_{G}\left(  A\right)  \right\vert .
\]

Thus, by Proposition \ref{prop17}, we obtain%
\[%
%TCIMACRO{\dsum \limits_{v\in V\left(  C\right)  }}%
%BeginExpansion
{\displaystyle\sum\limits_{v\in V\left(  C\right)  }}
%EndExpansion
\left(  \deg\left(  v\right)  -2\right)  \leq\left\vert N_{G}\left(
\mathrm{diadem}\left(  G\right)  \right)  -\mathrm{diadem}\left(  G\right)
\right\vert
\]
as claimed.
\end{proof}

It is worth mentioning that the inequality in Lemma \ref{lem13} may be strict
or may turn out to be an equality. For instance, consider the two graphs from
Figure \ref{fig34}, where:

\begin{itemize}
\item $\mathrm{diadem}\left(  G_{1}\right)  =\left\{  a,c,y,z\right\}  $,
$N_{G_{1}}\left(  \mathrm{diadem}\left(  G_{1}\right)  \right)  =\left\{
b,x,y,z\right\}  $ and, consequently,
\end{itemize}

\[%
%TCIMACRO{\dsum \limits_{v\in V\left(  C\right)  }}%
%BeginExpansion
{\displaystyle\sum\limits_{v\in V\left(  C\right)  }}
%EndExpansion
\left(  \deg\left(  v\right)  -2\right)  =1<2=\left\vert N_{G_{1}}\left(
\mathrm{diadem}\left(  G_{1}\right)  \right)  -\mathrm{diadem}\left(
G_{1}\right)  \right\vert ;
\]

\begin{figure}[h]
\setlength{\unitlength}{1cm}\begin{picture}(5,1.3)\thicklines
\multiput(2,0)(1,0){5}{\circle*{0.29}}
\multiput(2,1)(1,0){4}{\circle*{0.29}}
\put(2,0){\line(1,0){4}}
\put(2,0){\line(0,1){1}}
\put(2,1){\line(1,-1){1}}
\put(3,1){\line(1,0){2}}
\put(3,1){\line(1,-1){1}}
\put(4,0){\line(1,1){1}}
\put(3,1.3){\makebox(0,0){$a$}}
\put(4,1.33){\makebox(0,0){$b$}}
\put(5,1.3){\makebox(0,0){$c$}}
\put(4,0.35){\makebox(0,0){$x$}}
\put(5,0.35){\makebox(0,0){$y$}}
\put(6,0.35){\makebox(0,0){$z$}}
\put(1.2,0.5){\makebox(0,0){$G_{1}$}}
\multiput(8,0)(1,0){4}{\circle*{0.29}}
\multiput(8,1)(1,0){4}{\circle*{0.29}}
\put(8,0){\line(1,0){3}}
\put(8,0){\line(0,1){1}}
\put(8,1){\line(1,-1){1}}
\put(9,1){\line(1,-1){1}}
\put(10,0){\line(0,1){1}}
\put(11,0){\line(0,1){1}}
\put(9,1.3){\makebox(0,0){$u$}}
\put(10,1.3){\makebox(0,0){$v$}}
\put(11,1.3){\makebox(0,0){$w$}}
\put(10.3,0.3){\makebox(0,0){$s$}}
\put(11.3,0.3){\makebox(0,0){$t$}}
\put(7.2,0.5){\makebox(0,0){$G_{2}$}}
\end{picture}\caption{Both $G_{1}$ and $G_{2}$ are almost bipartite
non-K\"{o}nig-Egerv\'{a}ry.}%
\label{fig34}%
\end{figure}
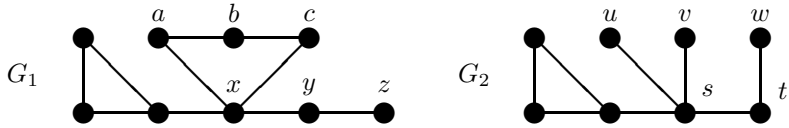

\begin{itemize}
\item $\mathrm{diadem}\left(  G_{2}\right)  =\left\{  u,v,w,t\right\}  $,
$N_{G_{2}}\left(  \mathrm{diadem}\left(  G_{2}\right)  \right)  =\left\{
s,t,w\right\}  $ and, consequently,
\end{itemize}

\[%
%TCIMACRO{\dsum \limits_{v\in V\left(  C\right)  }}%
%BeginExpansion
{\displaystyle\sum\limits_{v\in V\left(  C\right)  }}
%EndExpansion
\left(  \deg\left(  v\right)  -2\right)  =1=\left\vert N_{G_{2}}\left(
\mathrm{diadem}\left(  G_{2}\right)  \right)  -\mathrm{diadem}\left(
G_{2}\right)  \right\vert .
\]

The following theorem validates a conjecture from \cite{LevMan2023a}.

\begin{theorem}
\label{th18}If $G$ is an almost bipartite non-K\"{o}nig-Egerv\'{a}ry graph,
then every maximum matching of $G$ contains $\left\lfloor \frac{V(C)}%
{2}\right\rfloor $ edges belonging to its unique odd cycle $C$. Moreover,
every maximum matching of $G\left[  V\left(  C\right)  \right]  $ may be
enlarged to a maximum matching of $G$ that saturates all the vertices of $C$
but one.
\end{theorem}

\begin{proof}
Let $A\in\mathrm{MaxCritIndep}(G)$ and $M_{0}$ be a matching from
$N_{G}\left(  A\right)  $\ into $A$, that exists according to Theorem
\ref{th333}\emph{(iii)}.

Suppose that $M$ is a maximum matching of $G$ such that $\left\vert M\cap
E\left(  C\right)  \right\vert =\left\lfloor \frac{V\left(  C\right)  }%
{2}\right\rfloor -1$. Hence, there exist two adjacent vertices belonging to
$V\left(  C\right)  $, say $a$ and $b$, which are not saturated by $M$.
Clearly, at least one of them has a neighbor outside of $V\left(  C\right)  $,
otherwise $\left\vert M\cup\left\{  ab\right\}  \right\vert >\left\vert
M\right\vert $. By Proposition \ref{prop15}, Theorem \ref{conj1}, and Lemma
\ref{lem13},
\[
\left(  N_{G}\left(  a\right)  \cup N_{G}\left(  b\right)  \right)  -V\left(
C\right)  \subseteq N_{G}\left(  A\right)  \text{ and }\left(  N_{G}\left(
a\right)  \cap N_{G}\left(  b\right)  \right)  -V\left(  C\right)
=\emptyset.
\]

\begin{itemize}
\item Case 1. $\left\vert M\left(  a\right)  \right\vert =1$ and $\left\vert
M\left(  b\right)  \right\vert =1$. Thus, $M\left(  a\right)  ,M\left(
b\right)  \in N_{G}\left(  A\right)  $, and at most $\left\vert M_{0}%
\right\vert -2=\left\vert N_{G}\left(  A\right)  \right\vert -2$ edges of $M$
are located between $N_{G}\left(  A\right)  $\ and $A$. Therefore, by
Corollary \ref{cor8}, we obtain
\begin{gather*}
\left\vert M\right\vert \leq\frac{\left\vert V\left(  C\right)  \right\vert
-1}{2}-1+2+\left\vert N_{G}\left(  A\right)  \right\vert -2=\\
\frac{\left\vert V\left(  C\right)  \right\vert -1}{2}-1+\frac{n\left(
G\right)  -\left\vert V\left(  C\right)  \right\vert -d\left(  G\right)  }%
{2}=\mu\left(  G\right)  -1,
\end{gather*}
which contradicts the fact that $M$ is a maximum matching of $G$.

\item Case 2. $\left\vert M\left(  a\right)  \right\vert =1$ and $\left\vert
M\left(  b\right)  \right\vert =0$ or vice versa. Thus, $M\left(  a\right)
\in N_{G}\left(  A\right)  $, and at most $\left\vert M_{0}\right\vert
-1=\left\vert N_{G}\left(  A\right)  \right\vert -1$ edges of $M$ are located
between $N_{G}\left(  A\right)  $\ and $A$. Therefore, by Corollary
\ref{cor8}, we obtain
\begin{gather*}
\left\vert M\right\vert \leq\frac{\left\vert V\left(  C\right)  \right\vert
-1}{2}-1+1+\left\vert N_{G}\left(  A\right)  \right\vert -1=\\
\frac{\left\vert V\left(  C\right)  \right\vert -1}{2}-1+\frac{n\left(
G\right)  -\left\vert V\left(  C\right)  \right\vert -d\left(  G\right)  }%
{2}=\mu\left(  G\right)  -1,
\end{gather*}
which contradicts the fact that $M$ is a maximum matching of $G$.
\end{itemize}

If $\left\vert M\cap E\left(  C\right)  \right\vert <\left\lfloor
\frac{V\left(  C\right)  }{2}\right\rfloor -1$, then the gap between
$\left\vert M\right\vert $ and $\mu\left(  G\right)  $ is even bigger than $1$.

Therefore, if $M$ is a maximum matching of $G$, then $\left\vert M\cap
E\left(  C\right)  \right\vert =\left\lfloor \frac{V\left(  C\right)  }%
{2}\right\rfloor $ as stated.

In order to prove that every maximum matching of $G\left[  V\left(  C\right)
\right]  $ may be enlarged to a maximum matching of $G$, let us recall that
\[
\left\vert A\right\vert +\left\vert N_{G}\left(  A\right)  \right\vert
=n\left(  G\right)  -\left\vert V\left(  C\right)  \right\vert
\]
for every $A\in\mathrm{MaxCritIndep}(G)$ by Proposition \ref{prop15} and
Theorem \ref{conj1}. In addition, $\mu\left(  N_{G}\left[  A\right]  \right)
=\left\vert N_{G}\left(  A\right)  \right\vert $, since $G\left[  N_{G}\left[
A\right]  \right]  $ is a K\"{o}nig-Egerv\'{a}ry graph in accordance with
Theorem \ref{th100}\emph{(ii)}.

By Proposition \ref{prop15}, Theorem \ref{conj1}, and Corollary \ref{cor8}, it
is enough to prove that
\[
\frac{\left\vert V\left(  C\right)  \right\vert -1}{2}+\mu\left(  N_{G}\left[
A\right]  \right)  =\mu\left(  G\right)  =\frac{n\left(  G\right)  -d\left(
G\right)  -1}{2},
\]
which is true because%
\begin{gather*}
\frac{n\left(  G\right)  -d\left(  G\right)  -1}{2}=\\
\frac{\left\vert A\right\vert +\left\vert N_{G}\left(  A\right)  \right\vert
+\left\vert V\left(  C\right)  \right\vert -\left\vert A\right\vert
+\left\vert N_{G}\left(  A\right)  \right\vert -1}{2}=\frac{\left\vert
V\left(  C\right)  \right\vert -1}{2}+\left\vert N_{G}\left(  A\right)
\right\vert ,
\end{gather*}
as claimed.
\end{proof}

\begin{theorem}
If $G$ is an almost bipartite non-K\"{o}nig-Egerv\'{a}ry graph with the unique
odd cycle $C$, then each maximum matching of $G$ consists of $\ \left\lfloor
\frac{V\left(  C\right)  }{2}\right\rfloor $ edges from $C$, and either all
the edges of a maximum matching in $G\left[  N_{G}\left[  \mathrm{diadem}%
\left(  G\right)  \right]  \right]  $ or an edge $xy$ and a maximum matching
of $G\left[  N_{G}\left[  \mathrm{diadem}\left(  G\right)  \right]  -y\right]
$, where $x\in V\left(  C\right)  $ and $y\in N_{G}\left(  \mathrm{diadem}%
\left(  G\right)  \right)  $.
\end{theorem}

\begin{proof}
Let $A\in\mathrm{MaxCritIndep}(G)$. By Theorem \ref{th333}\textit{(iii)},
there exists a matching, say $M_{0}$, from $N_{G}\left(  A\right)  $ into $A$.

Since $A$ is critical, we know that $\left\vert A\right\vert -\left\vert
N_{G}\left(  A\right)  \right\vert =d\left(  G\right)  $.

Suppose that $M$ is a maximum matching of $G$. By Theorem \ref{th18},
$\left\vert M\cap E\left(  C\right)  \right\vert =\left\lfloor \frac{V\left(
C\right)  }{2}\right\rfloor $. Hence, there is only one vertex belonging to
$V\left(  C\right)  $, say $x$, that is not saturated by $M_{1}=M\cap E\left(
C\right)  $.

By Proposition \ref{prop15}\ and Theorem \ref{conj1}, we get$\ \left\vert
A\right\vert +\left\vert N_{G}\left(  A\right)  \right\vert =n\left(
G\right)  -\left\vert V\left(  C\right)  \right\vert $, because $N_{G}\left[
A\right]  =N_{G}\left[  \mathrm{diadem}\left(  G\right)  \right]  $ according
to Theorem \ref{th333}\textit{(ii) }and\textit{ }Theorem \ref{th100}%
\textit{(i).}

Thus
\[
\left\vert M_{0}\right\vert =\left\vert N_{G}\left(  A\right)  \right\vert
=\frac{n\left(  G\right)  -\left\vert V\left(  C\right)  \right\vert -d\left(
G\right)  }{2}.
\]

By Proposition \ref{prop15}, $M_{0}\cap M_{1}=\emptyset$. Therefore,
$M_{0}\cup M_{1}$ is a matching of $G$ and
\[
\left\vert M_{0}\cup M_{1}\right\vert =\left\vert M_{0}\right\vert +\left\vert
M_{1}\right\vert =\frac{n\left(  G\right)  -d\left(  G\right)  -\left\vert
V\left(  C\right)  \right\vert }{2}+\frac{\left\vert V\left(  C\right)
\right\vert -1}{2}=\frac{n\left(  G\right)  -d\left(  G\right)  -1}{2}%
\]
By Corollary \ref{cor8}, we infer that $M_{2}=M_{0}\cup M_{1}$ is a maximum
matching of $G$, because $\left\vert M_{2}\right\vert =\mu\left(  G\right)  $.
On the other hand,
\begin{gather}
\left\vert M\right\vert -\left\vert M_{1}\right\vert =\mu\left(  G\right)
-\left\lfloor \frac{V\left(  C\right)  }{2}\right\rfloor =\tag{*}\\
\frac{n\left(  G\right)  -d\left(  G\right)  -1}{2}-\frac{\left\vert V\left(
C\right)  \right\vert -1}{2}=\frac{n\left(  G\right)  -\left\vert V\left(
C\right)  \right\vert -d\left(  G\right)  }{2}=\left\vert N_{G}\left(
A\right)  \right\vert .\nonumber
\end{gather}

Clearly, $x$ is not saturated by $M_{2}$.

\textit{Case 1.} The vertex $x$ is not saturated by $M$. Hence, by Proposition
\ref{prop15} and Theorem \ref{conj1}, $M-M_{1}\subseteq\left(  N_{G}\left(
A\right)  ,A\right)  $. Since $A$ is an independent set, the equality (*)
ensures that $M-M_{1}$ is a maximum matching in $N_{G}\left[  A\right]  $.

\textit{Case 2.} The vertex $x$ is saturated by $M$. By Proposition
\ref{prop15}, $y=M\left(  x\right)  \in N_{G}\left(  A\right)  $. Hence, by
Proposition \ref{prop15} and Theorem \ref{conj1}, $M-M_{1}-\left\{
xy\right\}  \subseteq\left(  N_{G}\left(  A\right)  ,A\right)  $. Therefore,
$M-M_{1}-\left\{  xy\right\}  $ is a maximum matching in $G\left[
N_{G}\left[  A\right]  -y\right]  $, because $\left\vert M\right\vert
-\left\vert M_{1}\right\vert -1=\left\vert N_{G}\left(  A\right)  \right\vert
-1$ in accordance with the equality (*).
\end{proof}

Combining Theorem \ref{conj1} and Theorem \ref{th44}\emph{(i)} we get the following.

\begin{corollary}
\label{cor3}If $G$ is an almost bipartite non-K\"{o}nig-Egerv\'{a}ry graph
with the unique odd cycle $C$, then%
\[
V(C)\cup N_{G}\left[  \mathrm{diadem}\left(  G\right)  \right]
=\mathrm{corona}\left(  G\right)  \cup N_{G}\left(  \mathrm{core}\left(
G\right)  \right)  .
\]

\end{corollary}

\begin{proposition}
\cite{LevMan2022}\label{prop10} If $G$ is an almost bipartite
non-K\"{o}nig-Egerv\'{a}ry graph, then $\mathrm{core}\left(  G\right)  \cap
N_{G}\left[  V(C)\right]  =\emptyset$ (or, equivalently, $N_{G}\left[
\mathrm{core}\left(  G\right)  \right]  \cap V(C)=\emptyset$ by Corollary
\ref{cor11}\textit{)}.
\end{proposition}

\begin{corollary}
If $G$ is an almost bipartite non-K\"{o}nig-Egerv\'{a}ry graph, then
$V(C)\subseteq\mathrm{corona}\left(  G\right)  $, and, consequently,
$\left\vert V(C)\right\vert \leq2\alpha\left(  G\right)  +1-\xi\left(
G\right)  $.
\end{corollary}

\begin{proof}
An interaction between Corollary \ref{cor3} and Proposition \ref{prop10} takes
us to the inclusion $V(C)\subseteq\emph{corona}\left(  G\right)  $. Hence, by
Theorem \ref{th44}\emph{(iii)}, we obtain%
\[
\left\vert V(C)\right\vert \leq2\alpha\left(  G\right)  +1-\xi\left(
G\right)  \text{,}%
\]
as required.
\end{proof}

\section{Equalities and inequalities involving $\varrho_{v}\left(  G\right)
$}

\begin{proposition}
For every vertex $v\in V\left(  G\right)  $, if there exists a critical
independent set in $G-v$ that is not critical in $G$, then $d\left(  G\right)
\geq d\left(  G-v\right)  $.
\end{proposition}

\begin{proof}
Let $B$ be an independent set in $G-v$. Hence,
\[
d_{G}\left(  B\right)  =\left\vert B\right\vert -\left\vert N_{G}\left(
B\right)  \right\vert \leq\left\vert B\right\vert -\left\vert N_{G-v}\left(
B\right)  \right\vert =d_{G-v}\left(  B\right)  \leq d_{G}\left(  B\right)
+1,
\]
because
\[
\left\vert N_{G}\left(  B\right)  \right\vert \geq\left\vert N_{G-v}\left(
B\right)  \right\vert \geq\left\vert N_{G}\left(  B\right)  \right\vert -1.
\]

Suppose that $B$ is a critical independent set in $G-v$, i.e., $d_{G-v}\left(
B\right)  =d\left(  G-v\right)  $.

If $B$ is not critical in $G$, then $d\left(  G\right)  >d_{G}\left(
B\right)  $. Consequently, we get
\[
d\left(  G\right)  \geq d_{G}\left(  B\right)  +1\geq d_{G-v}\left(  B\right)
=d\left(  G-v\right)  ,
\]
as claimed.
\end{proof}

\begin{theorem}
\cite{Larson2007}\label{th3} Each critical independent set is contained in
some maximum critical independent set.
\end{theorem}

\begin{proposition}
If $v\in\mathrm{nucleus}\left(  G\right)  $ and there exists an independent
set that is critical in both $G$ and $G-v$, then $d\left(  G\right)  =d\left(
G-v\right)  $.
\end{proposition}

\begin{proof}
If $B$ is critical in $G$, then $B$ is not a maximum critical independent set
in $G$, because $v\in\mathrm{nucleus}\left(  G\right)  $, while $v\notin B$.
By Theorem \ref{th3}, there exists some $A\in\mathrm{MaxCritIndep}(G)$, such
that $B\subset A$. Since $v\in\mathrm{nucleus}\left(  G\right)  \subseteq A$,
we infer that $v\notin N_{G}\left(  B\right)  $, because $A$ is independent.
Therefore,
\[
d\left(  G-v\right)  =d_{G-v}\left(  B\right)  =\left\vert B\right\vert
-\left\vert N_{G-v}\left(  B\right)  \right\vert =\left\vert B\right\vert
-\left\vert N_{G}\left(  B\right)  \right\vert =d_{G}\left(  B\right)
=d\left(  G\right)  ,
\]
because $N_{G}\left(  B\right)  =N_{G-v}\left(  B\right)  $.
\end{proof}

\begin{lemma}
\label{lem8}If $v\notin N_{G}\left(  \mathrm{diadem}\left(  G\right)  \right)
$, then $d\left(  G\right)  \geq d\left(  G-v\right)  $.
\end{lemma}

\begin{proof}
For every independent set $B$ in $G-v$:

\begin{itemize}
\item $v\notin N_{G}\left(  B\right)  $ implies
\[
d_{G-v}\left(  B\right)  \,=\left\vert B\right\vert -\left\vert N_{G-v}\left(
B\right)  \right\vert =\left\vert B\right\vert -\left\vert N_{G}\left(
B\right)  \right\vert =d_{G}\left(  B\right)  \leq d\left(  G\right)  ,
\]
because $N_{G-v}\left(  B\right)  =N_{G}\left(  B\right)  $;

\item $v\in N_{G}\left(  B\right)  $ implies
\[
d_{G}\left(  B\right)  +1=\left\vert B\right\vert -\left\vert N_{G}\left(
B\right)  \right\vert +1=\left\vert B\right\vert -\left\vert N_{G-v}\left(
B\right)  \right\vert =d_{G-v}\left(  B\right)  ,
\]
because $\left\vert N_{G-v}\left(  B\right)  \right\vert +1=\left\vert
N_{G}\left(  B\right)  \right\vert $. The fact that $v\notin N_{G}\left(
\mathrm{diadem}\left(  G\right)  \right)  $ means that $v$ does not belong to
a neighborhood of any critical independent set, i.e., $B$ is not critical in
$G$. Thus $d\left(  G\right)  \geq d_{G}\left(  B\right)  +1=d_{G-v}\left(
B\right)  $.
\end{itemize}

Consequently, $d\left(  G\right)  \geq d\left(  G-v\right)  $ in the above two cases.
\end{proof}

It worth mentioning that there are graphs with $v\in N_{G}\left(
\mathrm{diadem}\left(  G\right)  \right)  $ and $d\left(  G\right)  <d\left(
G-v\right)  $. For instance, consider the graph $G_{2}$ from Figure
\ref{fig1}: $\mathrm{diadem}\left(  G_{2}\right)  =\left\{  y,z\right\}  $,
$d(G)=1>d(G-y)=0$, while $x\in N_{G}\left(  \mathrm{diadem}\left(  G\right)
\right)  $\ and $d(G)=1<d(G-x)=2$.

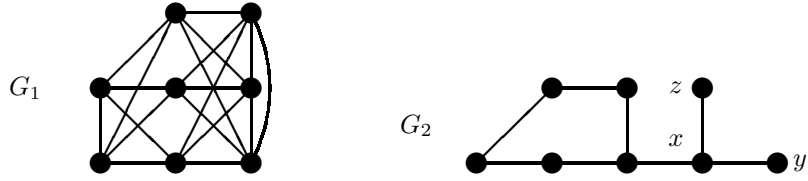
\begin{figure}[h]
\setlength{\unitlength}{1cm}\begin{picture}(5,2.2)\thicklines
\multiput(3,0)(1,0){3}{\circle*{0.29}}
\multiput(3,1)(1,0){3}{\circle*{0.29}}
\multiput(4,2)(1,0){2}{\circle*{0.29}}
\put(3,0){\line(1,0){2}}
\put(3,0){\line(0,1){1}}
\put(3,0){\line(1,1){1}}
\put(3,0){\line(1,2){1}}
\put(3,1){\line(1,1){1}}
\put(3,1){\line(1,-1){1}}
\put(3,1){\line(1,0){1}}
\put(4,0){\line(1,1){1}}
\put(4,0){\line(1,2){1}}
\put(4,1){\line(1,1){1}}
\put(4,1){\line(1,-1){1}}
\put(4,1){\line(1,0){1}}
\put(4,2){\line(1,0){1}}
\put(4,2){\line(1,-1){1}}
\put(4,2){\line(1,-2){1}}
\put(5,0){\line(0,1){2}}
\qbezier(5,0)(5.5,1)(5,2)
\put(2,1){\makebox(0,0){$G_{1}$}}
\multiput(8,0)(1,0){5}{\circle*{0.29}}
\multiput(9,1)(1,0){3}{\circle*{0.29}}
\put(8,0){\line(1,0){4}}
\put(8,0){\line(1,1){1}}
\put(9,1){\line(1,0){1}}
\put(10,0){\line(0,1){1}}
\put(11,0){\line(0,1){1}}
\put(10.65,0.3){\makebox(0,0){$x$}}
\put(12.3,0){\makebox(0,0){$y$}}
\put(10.65,1){\makebox(0,0){$z$}}
\put(7.2,0.5){\makebox(0,0){$G_{2}$}}
\end{picture}\caption{$1$-K\"{o}nig-Egerv\'{a}ry graphs.}%
\label{fig1}%
\end{figure}

\begin{proposition}
\cite{LevMan2024a}\label{prop11} If $A$ is a critical independent set in $G$,
then $\mathrm{core}\left(  G\right)  \cap N\left(  A\right)  =\emptyset$ and,
consequently, $\mathrm{core}\left(  G\right)  \cap N\left(  \mathrm{diadem}%
\left(  G\right)  \right)  =\emptyset$.
\end{proposition}

\begin{theorem}
\label{th5}If $G$ is an almost bipartite non-K\"{o}nig-Egerv\'{a}ry graph,
then
\begin{align*}
\varrho_{v}\left(  G\right)   &  =n\left(  G\right)  +d\left(  G\right)
-\xi\left(  G\right)  -\beta(G)\\
&  =2\alpha\left(  G\right)  +1-\xi\left(  G\right)  -\beta(G).
\end{align*}

\end{theorem}

\begin{proof}
By Lemma \ref{lem84}, an almost bipartite non-K\"{o}nig-Egerv\'{a}ry graph is
$1$-K\"{o}nig-Egerv\'{a}ry, i.e.,
\[
\alpha\left(  G\right)  +\mu\left(  G\right)  =n\left(  G\right)  -1.
\]

Clearly, $G-v$ is either bipartite or almost bipartite, for every $v\in
V\left(  G\right)  $. By Theorems \ref{th44}(\textit{ii}) and \ref{th43}%
(\textit{ii}), we know that
\[
d\left(  G\right)  =\alpha\left(  G\right)  -\mu\left(  G\right)  \text{ and
}d\left(  G-v\right)  =\alpha\left(  G-v\right)  -\mu\left(  G-v\right)  .
\]

By Lemma \ref{lem8}, if $v\notin N\left(  \mathrm{diadem}\left(  G\right)
\right)  $, then $d\left(  G\right)  \geq d\left(  G-v\right)  $.

Therefore, if $v\notin N\left(  \mathrm{diadem}\left(  G\right)  \right)  $,
then%
\[
d\left(  G\right)  =\alpha\left(  G\right)  -\mu\left(  G\right)  \geq
\alpha\left(  G-v\right)  -\mu\left(  G-v\right)  =d\left(  G-v\right)  .
\]
Hence, $\mu\left(  G-v\right)  \geq\mu\left(  G\right)  $. On the other hand,
$\mu\left(  G\right)  \geq\mu\left(  G-v\right)  $ in general. Consequently,
$\mu\left(  G\right)  =\mu\left(  G-v\right)  $.

In addition, if $v\notin\mathrm{core}\left(  G\right)  $, then $\alpha\left(
G\right)  =\alpha\left(  G-v\right)  $. Hence, if $v\notin\mathrm{core}\left(
G\right)  $ and $v\notin N\left(  \mathrm{diadem}\left(  G\right)  \right)  $,
then
\[
\alpha\left(  G-v\right)  +\mu\left(  G-v\right)  =\alpha\left(  G\right)
+\mu\left(  G\right)  =n\left(  G\right)  -1=n\left(  G-v\right)  ,
\]
which means that $G-v$ is a K\"{o}nig-Egerv\'{a}ry graph. Consequently, we
get
\[
\varrho_{v}\left(  G\right)  \geq n\left(  G\right)  -\xi\left(  G\right)
-\left\vert N\left(  \mathrm{diadem}\left(  G\right)  \right)  \right\vert ,
\]
since, by Proposition \ref{prop11}, $\mathrm{core}\left(  G\right)  \cap
N\left(  \mathrm{diadem}\left(  G\right)  \right)  =\emptyset$. On the other
hand, Theorem \ref{th10} claims that
\[
\varrho_{v}\left(  G\right)  \leq n\left(  G\right)  -\xi\left(  G\right)
-\left\vert N\left(  \mathrm{diadem}\left(  G\right)  \right)  \right\vert
\]
for $1$--K\"{o}nig-Egerv\'{a}ry graphs. All in all, $\varrho_{v}\left(
G\right)  =n\left(  G\right)  -\xi\left(  G\right)  -\left\vert N\left(
\mathrm{diadem}\left(  G\right)  \right)  \right\vert $.

By Theorem \ref{th444}\textit{(ii)}, the set $\mathrm{diadem}\left(  G\right)
$ is critical in $G$. Hence, $d\left(  G\right)  =\left\vert \mathrm{diadem}%
\left(  G\right)  \right\vert -\left\vert N\left(  \mathrm{diadem}\left(
G\right)  \right)  \right\vert $. Thus,%
\[
\varrho_{v}\left(  G\right)  =n\left(  G\right)  +d\left(  G\right)
-\xi\left(  G\right)  -\beta(G),
\]
and together with Corollary \ref{cor8} this completes the proof.
\end{proof}

Let $A\in\mathrm{MaxCritIndep}(G)$. Then $A$ $\subseteq\mathrm{diadem}\left(
G\right)  $, and, therefore,
\[
\beta\left(  G\right)  =\left\vert \mathrm{diadem}\left(  G\right)
\right\vert \geq\left\vert A\right\vert =\alpha^{\prime}\left(  G\right)  .
\]
Consequently, by Theorem \ref{th5}, we arrive at the following.

\begin{corollary}
If $G$ is an almost bipartite non-K\"{o}nig-Egerv\'{a}ry graph, then
\[
\left\vert V\left(  C\right)  \right\vert \leq\varrho_{v}\left(  G\right)
\leq2\alpha\left(  G\right)  -\xi\left(  G\right)  -\alpha^{\prime}\left(
G\right)  +1
\]
where $C$ is the unique odd cycle in $G$.
\end{corollary}

\begin{corollary}
\label{cor2}If $G$ is an almost bipartite non-K\"{o}nig-Egerv\'{a}ry graph
with the unique odd cycle $C$, then
\[
\varrho_{v}\left(  G\right)  =\left\vert V\left(  C\right)  \right\vert
+\left\vert \mathrm{nucleus}\left(  G\right)  \right\vert -\left\vert
\mathrm{core}\left(  G\right)  \right\vert .
\]

\end{corollary}

\begin{proof}
By Theorem \ref{conj1}, we know that
\[
n\left(  G\right)  =\left\vert V\left(  C\right)  \right\vert +\left\vert
N\left[  \mathrm{diadem}\left(  G\right)  \right]  \right\vert .
\]

Together with Theorem \ref{th5} and Proposition \ref{prop14} we conclude that
\begin{gather*}
\varrho_{v}\left(  G\right)  =n\left(  G\right)  +d\left(  G\right)
-\xi\left(  G\right)  -\beta(G)=\\
\left\vert V\left(  C\right)  \right\vert +\left\vert N\left[  \mathrm{diadem}%
\left(  G\right)  \right]  \right\vert +d\left(  G\right)  -\xi\left(
G\right)  -\beta(G)=\\
\left\vert V\left(  C\right)  \right\vert +\beta(G)+\left\vert
\mathrm{nucleus}\left(  G\right)  \right\vert -d\left(  G\right)  +d\left(
G\right)  -\xi\left(  G\right)  -\beta(G)=\\
\left\vert V\left(  C\right)  \right\vert +\left\vert \mathrm{nucleus}\left(
G\right)  \right\vert -\left\vert \mathrm{core}\left(  G\right)  \right\vert ,
\end{gather*}
as required.
\end{proof}

\begin{corollary}
\label{cor5}If $G$ is an almost bipartite non-K\"{o}nig-Egerv\'{a}ry graph
with the unique odd cycle $C$, then
\begin{gather*}
\left\vert V\left(  C\right)  \right\vert +\left\vert \mathrm{nucleus}\left(
G\right)  \right\vert +\left\vert \mathrm{diadem}\left(  G\right)  \right\vert
=\\
2\alpha\left(  G\right)  +1=\left\vert \mathrm{core}\left(  G\right)
\right\vert +\left\vert \mathrm{corona}\left(  G\right)  \right\vert .
\end{gather*}

\end{corollary}

\begin{proof}
By Corollary \ref{cor2}, we know that%
\[
\varrho_{v}\left(  G\right)  =\left\vert V\left(  C\right)  \right\vert
+\left\vert \mathrm{nucleus}\left(  G\right)  \right\vert -\left\vert
\mathrm{core}\left(  G\right)  \right\vert .
\]
On the other hand, by Theorem \ref{th5}, we obtain
\[
\varrho_{v}\left(  G\right)  =2\alpha\left(  G\right)  +1-\xi\left(  G\right)
-\left\vert \mathrm{diadem}\left(  G\right)  \right\vert .
\]
Thus,%
\[
\left\vert V\left(  C\right)  \right\vert +\left\vert \mathrm{nucleus}\left(
G\right)  \right\vert -\xi\left(  G\right)  =2\alpha\left(  G\right)
+1-\xi\left(  G\right)  -\left\vert \mathrm{diadem}\left(  G\right)
\right\vert ,
\]
which means%
\[
\left\vert V\left(  C\right)  \right\vert +\left\vert \mathrm{nucleus}\left(
G\right)  \right\vert +\left\vert \mathrm{diadem}\left(  G\right)  \right\vert
=2\alpha\left(  G\right)  +1.
\]
as claimed.

To complete the proof use Theorem \ref{th44}\emph{(iii)}.
\end{proof}

\begin{theorem}
\label{cor7}If $G$ is an almost bipartite non-K\"{o}nig-Egerv\'{a}ry graph,
then
\[
\varrho_{v}\left(  G\right)  =\left\vert \mathrm{corona}\left(  G\right)
\right\vert -\left\vert \mathrm{diadem}\left(  G\right)  \right\vert .
\]

\end{theorem}

\begin{proof}
By Corollary \ref{cor5} and Corollary \ref{cor2}, we obtain
\[
\varrho_{v}\left(  G\right)  =\left\vert V\left(  C\right)  \right\vert
+\left\vert \mathrm{nucleus}\left(  G\right)  \right\vert -\left\vert
\mathrm{core}\left(  G\right)  \right\vert =\left\vert \mathrm{corona}\left(
G\right)  \right\vert -\left\vert \mathrm{diadem}\left(  G\right)  \right\vert
,
\]
as required.
\end{proof}

\begin{corollary}
\label{cor10}If $G$ is an almost bipartite graph, then,
\[
\left\vert \mathrm{core}\left(  G\right)  \right\vert +\left\vert
\mathrm{corona}\left(  G\right)  \right\vert =n\left(  G\right)  +d\left(
G\right)  .
\]

\end{corollary}

\begin{proof}
Suppose that $G$ is almost bipartite.

\textit{Case 1.} $G$ is non-K\"{o}nig-Egerv\'{a}ry. By Theorem \ref{th5} and
Theorem \ref{cor7}, we know that
\[
n\left(  G\right)  +d\left(  G\right)  -\xi\left(  G\right)  -\beta
(G)=\left\vert \mathrm{corona}\left(  G\right)  \right\vert -\left\vert
\mathrm{diadem}\left(  G\right)  \right\vert .
\]
Hence,
\[
\left\vert \mathrm{corona}\left(  G\right)  \right\vert =n\left(  G\right)
+d\left(  G\right)  -\xi\left(  G\right)
\]
as claimed.

\textit{Case 2.} $G$ is K\"{o}nig-Egerv\'{a}ry. By the definition of
K\"{o}nig-Egerv\'{a}ry graphs and Theorem \ref{th43}\emph{(ii)}, we obtain
$2\alpha\left(  G\right)  =$ $n\left(  G\right)  +d\left(  G\right)  $. To
complete the proof use Theorem \ref{th43}\emph{(iii).}
\end{proof}

Let us mention that $\mathrm{corona}(G)\cup N_{G}\left(  \mathrm{core}%
(G)\right)  =V(G)$ is true for every almost bipartite graph in accordance with
Theorem \ref{th43}\emph{(i) }and\emph{ }Theorem \ref{th44}\emph{(i)}.

\begin{theorem}
\cite{LevMan2024a}\label{lem11} Suppose that $G$\ is a $1$%
-\textit{K\"{o}nig-Egerv\'{a}ry graph. Then }$\varrho_{v}\left(  G\right)
=n\left(  G\right)  $ if and only if $\mu\left(  G\right)  <\frac{n\left(
G\right)  }{2}$, $\xi\left(  G\right)  =0$ and $\beta\left(  G\right)  =0$.
\end{theorem}

\begin{proposition}
\label{prop18}If $G$ is an almost bipartite non-K\"{o}nig-Egerv\'{a}ry graph,
then $\varrho_{v}\left(  G\right)  =n\left(  G\right)  $ if and only if
$G=C_{2k+1}$ for some integer $k\geq1$.
\end{proposition}

\begin{proof}
Suppose $G$ is an almost bipartite non-K\"{o}nig-Egerv\'{a}ry graph and
$\varrho_{v}\left(  G\right)  =n\left(  G\right)  $.

By definitions of $\mathrm{nucleus}(G)$ and $\mathrm{diadem}(G)$,%
\begin{align*}
\mathrm{nucleus}(G)  &  =\bigcap\{S:S\in\mathrm{MaxCritIndep}(G)\}\subseteq\\
\bigcup\{S  &  :S\in\mathrm{MaxCritIndep}(G)\}=\mathrm{diadem}(G).
\end{align*}
Hence, the fact that $\beta\left(  G\right)  =0$, i.e., $\left\vert
\mathrm{diadem}\left(  G\right)  \right\vert =0$, implies $\left\vert
\mathrm{nucleus}\left(  G\right)  \right\vert =0$.

Now, by Corollary \ref{cor2}, claiming that
\[
\varrho_{v}\left(  G\right)  =\left\vert V\left(  C\right)  \right\vert
+\left\vert \mathrm{nucleus}\left(  G\right)  \right\vert -\left\vert
\mathrm{core}\left(  G\right)  \right\vert ,
\]
and Theorem \ref{lem11} we conclude that $n\left(  G\right)  =\varrho
_{v}\left(  G\right)  =\left\vert V\left(  C\right)  \right\vert $. In other
words, $G=C_{2k+1}$ for some integer $k\geq1$.

The converse is clear.
\end{proof}

Proposition \ref{prop18} immediately implies the following.

\begin{corollary}
If $G$ is an almost bipartite non-K\"{o}nig-Egerv\'{a}ry graph different from
an odd cycle, then there exists a vertex $v\in V\left(  G\right)  $ such that
$G-v$ is not K\"{o}nig-Egerv\'{a}ry.
\end{corollary}

Let us recall that $V(C)\cap N_{G}\left[  \mathrm{diadem}\left(  G\right)
\right]  =\emptyset$. The graphs from Figure \ref{fig2} have:

$\varrho_{v}\left(  G_{1}\right)  =\left\vert V\left(  C_{1}\right)
\right\vert +\left\vert \mathrm{nucleus}\left(  G_{1}\right)  \right\vert
-\xi\left(  G_{1}\right)  =4$ and

$\varrho_{v}\left(  G_{2}\right)  =\left\vert V\left(  C_{2}\right)
\right\vert +\left\vert \mathrm{nucleus}\left(  G_{2}\right)  \right\vert
-\xi\left(  G_{2}\right)  =\left\vert V\left(  C_{2}\right)  \right\vert =3$.

\begin{figure}[h]
\setlength{\unitlength}{1cm}\begin{picture}(5,1.2)\thicklines
\multiput(3,0)(1,0){3}{\circle*{0.29}}
\multiput(4,1)(1,0){2}{\circle*{0.29}}
\put(3,0){\line(1,0){2}}
\put(3,0){\line(1,1){1}}
\put(4,0){\line(0,1){1}}
\put(5,0){\line(0,1){1}}
\put(5.35,1){\makebox(0,0){$x$}}
\put(2,0.5){\makebox(0,0){$G_{1}$}}
\multiput(8,0)(1,0){4}{\circle*{0.29}}
\multiput(9,1)(1,0){2}{\circle*{0.29}}
\put(8,0){\line(1,0){3}}
\put(8,0){\line(1,1){1}}
\put(9,0){\line(0,1){1}}
\put(10,0){\line(0,1){1}}
\put(10.35,1){\makebox(0,0){$u$}}
\put(11.35,0){\makebox(0,0){$v$}}
\put(7,0.5){\makebox(0,0){$G_{2}$}}
\end{picture}\caption{$\mathrm{core}\left(  G_{1}\right)  \subset
\mathrm{nucleus}\left(  G_{1}\right)  =\left\{  x\right\}  $, while
$\mathrm{core}\left(  G_{2}\right)  =\mathrm{nucleus}\left(  G_{2}\right)
=\left\{  u,v\right\}  $}%
\label{fig2}%
\end{figure}
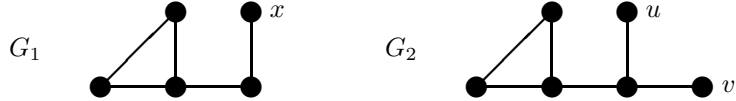

\begin{theorem}
If $G$ is an almost bipartite non-K\"{o}nig-Egerv\'{a}ry graph with the unique
odd cycle $C$, then
\[
\varrho_{v}\left(  G\right)  \geq%
%TCIMACRO{\dsum \limits_{v\in V\left(  C\right)  }}%
%BeginExpansion
{\displaystyle\sum\limits_{v\in V\left(  C\right)  }}
%EndExpansion
\deg\left(  v\right)  -\left\vert V\left(  C\right)  \right\vert -\left\vert
N_{G}\left(  V(C)\right)  \cap N_{G}\left(  \mathrm{core}\left(  G\right)
\right)  \right\vert .
\]

\end{theorem}

\begin{proof}
Clearly, if $v\in V\left(  C\right)  $, then $G-v$ is bipartite, and
consequently, K\"{o}nig-Egerv\'{a}ry. Thus $\left\vert V\left(  C\right)
\right\vert \leq\varrho_{v}\left(  G\right)  $.

By Lemma \ref{lem13}, the number of different neighbors of the cycle $C$ not
belonging to $C$ is$%
%TCIMACRO{\dsum \limits_{v\in V\left(  C\right)  }}%
%BeginExpansion
{\displaystyle\sum\limits_{v\in V\left(  C\right)  }}
%EndExpansion
\left(  \deg\left(  v\right)  -2\right)  $, since $\left(  N_{G}\left(
x\right)  \cap N_{G}\left(  y\right)  \right)  -V\left(  C\right)  =\emptyset$
for every distinct $x,y\in V\left(  C\right)  $. Moreover, if $b\in
N_{G}\left(  V\left(  C\right)  \right)  -V\left(  C\right)  $, then, by Lemma
\ref{lem13}, Theorem \ref{th100}\emph{(i}), Proposition\emph{ }\ref{prop15}
and Theorem \ref{conj1}, there exist: a unique vertex $a\in$ $V\left(
C\right)  \cap N_{G}\left(  b\right)  $, and some $A\in\mathrm{MaxCritIndep}%
(G)$, such that $b\in N_{G}\left(  A\right)  $.

Theorem \ref{th18} allows us to start with a maximum matching of $G$, say $M$,
not saturating $a$ but saturating $b$. Let us define a new matching $M_{1}$ as
follows: $M_{1}\left(  a\right)  =b$, $M_{1}\left(  b\right)  =a$, and
$M_{1}\left(  u\right)  =M\left(  u\right)  $ for every $u\in V\left(
G\right)  -\left\{  a,b\right\}  $. Hence, $M\left(  b\right)  $ is not $\mu
$-critical. In addition, if $b\notin N_{G}\left(  V(C)\right)  \cap
N_{G}\left(  \mathrm{core}\left(  G\right)  \right)  $, then $M\left(
b\right)  $ is not $\alpha$-critical. Thus, if $b\notin N_{G}\left(
V(C)\right)  \cap N_{G}\left(  \mathrm{core}\left(  G\right)  \right)  $, then%
\[
\alpha\left(  G-M\left(  b\right)  \right)  +\mu\left(  G-M\left(  b\right)
\right)  =\alpha\left(  G\right)  +\mu\left(  G\right)  =n\left(  G\right)
-1=n\left(  G-M\left(  b\right)  \right)  ,
\]
i.e., $G-M\left(  b\right)  $ is K\"{o}nig-Egerv\'{a}ry. Therefore,%
\[
\varrho_{v}\left(  G\right)  \geq%
%TCIMACRO{\dsum \limits_{v\in V\left(  C\right)  }}%
%BeginExpansion
{\displaystyle\sum\limits_{v\in V\left(  C\right)  }}
%EndExpansion
\left(  \deg\left(  v\right)  -2\right)  -\left\vert N_{G}\left(  V(C)\right)
\cap N_{G}\left(  \mathrm{core}\left(  G\right)  \right)  \right\vert .
\]
It is worth pointing out that the above inequality is based on the vertices
belonging to $N_{G}\left(  V\left(  C\right)  \right)  -V\left(  C\right)  $.
On the other hand, $G-x$ is K\"{o}nig-Egerv\'{a}ry for every $x\in V\left(
C\right)  $, because $G-x$ is bipartite. Consequently,
\begin{gather*}
\varrho_{v}\left(  G\right)  \geq\left\vert V\left(  C\right)  \right\vert +%
%TCIMACRO{\dsum \limits_{v\in V\left(  C\right)  }}%
%BeginExpansion
{\displaystyle\sum\limits_{v\in V\left(  C\right)  }}
%EndExpansion
\left(  \deg\left(  v\right)  -2\right)  -\left\vert N_{G}\left(  V(C)\right)
\cap N_{G}\left(  \mathrm{core}\left(  G\right)  \right)  \right\vert =\\%
%TCIMACRO{\dsum \limits_{v\in V\left(  C\right)  }}%
%BeginExpansion
{\displaystyle\sum\limits_{v\in V\left(  C\right)  }}
%EndExpansion
\deg\left(  v\right)  -\left\vert V\left(  C\right)  \right\vert -\left\vert
N_{G}\left(  V(C)\right)  \cap N_{G}\left(  \mathrm{core}\left(  G\right)
\right)  \right\vert ,
\end{gather*}
as stated.
\end{proof}

\begin{corollary}
\label{cor13}If $G$ is an almost bipartite non-K\"{o}nig-Egerv\'{a}ry graph
with the unique odd cycle $C$ such that $N_{G}\left(  V(C)\right)  \cap
N_{G}\left(  \mathrm{core}\left(  G\right)  \right)  =\emptyset$, then
\[
\varrho_{v}\left(  G\right)  \geq%
%TCIMACRO{\dsum \limits_{v\in V\left(  C\right)  }}%
%BeginExpansion
{\displaystyle\sum\limits_{v\in V\left(  C\right)  }}
%EndExpansion
\deg\left(  v\right)  -\left\vert V\left(  C\right)  \right\vert .
\]

\end{corollary}

The graph $G$ in Figure \ref{fig11222} shows that the inequality from
Corollary \ref{cor13} may be strict.

\begin{figure}[h]
\setlength{\unitlength}{1cm}\begin{picture}(5,1.2)\thicklines
\multiput(5,0)(1,0){5}{\circle*{0.29}}
\multiput(5,1)(1,0){5}{\circle*{0.29}}
\put(5,0){\line(1,0){4}}
\put(5,0){\line(0,1){1}}
\put(5,1){\line(1,-1){1}}
\put(6,1){\line(1,0){2}}
\put(6,1){\line(1,-1){1}}
\put(7,0){\line(1,1){1}}
\put(8,0){\line(1,1){1}}
\put(4.2,0.5){\makebox(0,0){$G$}}
\end{picture}\caption{$\alpha\left(  G\right)  =5$, $\mu\left(  G\right)  =4$,
$\varrho_{v}\left(  G\right)  =5$, and ${\displaystyle\sum\limits_{v\in
V\left(  C\right)  }}\deg\left(  v\right)  -\left\vert V\left(  C\right)
\right\vert =4$}%
\label{fig11222}%
\end{figure}
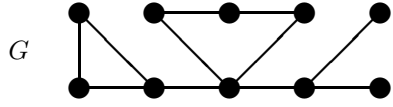

\begin{corollary}
If $G$ is a connected almost bipartite non-K\"{o}nig-Egerv\'{a}ry graph with
the unique odd cycle $C$ such that $n\left(  G\right)  >\left\vert V\left(
C\right)  \right\vert $ and $N_{G}\left(  V(C)\right)  \cap N_{G}\left(
\mathrm{core}\left(  G\right)  \right)  =\emptyset$, then there exists a
vertex $v\in V\left(  G\right)  -$ $V\left(  C\right)  $ such that $G-v$ is K\"{o}nig-Egerv\'{a}ry.
\end{corollary}

It is worth mentioning that there exist a connected almost bipartite
non-K\"{o}nig-Egerv\'{a}ry graph with the unique cycle $C$ such that $n\left(
G\right)  >\left\vert V\left(  C\right)  \right\vert $ and $\varrho_{v}\left(
G\right)  =\left\vert V\left(  C\right)  \right\vert $, for instance, see the
graph $G_{2}$ in Figure \ref{fig2}. In general, $\left\vert V\left(  C\right)
\right\vert <\varrho_{v}\left(  G\right)  $ if and only if $\left\vert
\mathrm{core}\left(  G\right)  \right\vert <\left\vert \mathrm{nucleus}\left(
G\right)  \right\vert $, in accordance with Corollary \ref{cor2}.

\section{Conclusions}

Theorem \ref{th5} claims that if $G$ is an almost bipartite
non-K\"{o}nig-Egerv\'{a}ry graph, then
\[
\varrho_{v}\left(  G\right)  =n\left(  G\right)  +d\left(  G\right)
-\xi\left(  G\right)  -\beta\left(  G\right)  ,
\]
while for general $1$-K\"{o}nig-Egerv\'{a}ry graphs $n\left(  G\right)
+d\left(  G\right)  -\xi\left(  G\right)  -\beta\left(  G\right)  $ is only an
upper bound for $\varrho_{v}\left(  G\right)  $ in accordance with Theorem
\ref{th10}. It motivates the following.

\begin{problem}
Characterize $1$-K\"{o}nig-Egerv\'{a}ry graphs such that
\[
\varrho_{v}\left(  G\right)  =n\left(  G\right)  +d\left(  G\right)
-\xi\left(  G\right)  -\beta\left(  G\right)  .
\]

\end{problem}

Theorem \ref{th2222} states that $\mathrm{\ker}(G)=\mathrm{core}(G)$ for each
almost bipartite non-K\"{o}nig-Egerv\'{a}ry graph. Hence, by definitions of
$\mathrm{\ker}(G)$ and $\mathrm{nucleus}(G)$, we see that%
\begin{gather*}
\mathrm{core}\left(  G\right)  =\mathrm{\ker}(G)=\\
=\bigcap\{S:S\in\mathrm{CritIndep}(G)\}\subseteq\bigcap\{S:S\in
\mathrm{MaxCritIndep}(G)\}=\\
=\mathrm{nucleus}(G)
\end{gather*}
for such graphs. Together with Corollary \ref{cor2}, it allows us to conclude
with the following.

\begin{corollary}
Let $G$ be an almost bipartite non-K\"{o}nig-Egerv\'{a}ry graph with the
unique odd cycle $C$. Then $\varrho_{v}\left(  G\right)  =\left\vert V\left(
C\right)  \right\vert $ if and only if $\mathrm{core}\left(  G\right)
=\mathrm{nucleus}\left(  G\right)  $.
\end{corollary}

It is known that, in general, $\mathrm{core}\left(  G\right)  $ does not have
to be a subset of $\mathrm{nucleus}\left(  G\right)  $ and vice versa
\cite{Short2016}. It directs us to the following.

\begin{problem}
Characterize graphs such that $\left\vert \mathrm{nucleus}\left(  G\right)
\right\vert =\left\vert \mathrm{core}\left(  G\right)  \right\vert $.
\end{problem}

In this paper, most of our findings deal with $\varrho_{v}\left(  G\right)  $
for almost bipartite non-K\"{o}nig-Egerv\'{a}ry graphs. It was shown in
\cite{LevMan2024a} that if $G$ is neither K\"{o}nig-Egerv\'{a}ry nor
$1$-K\"{o}nig-Egerv\'{a}ry, then $\varrho_{e}\left(  G\right)  =0$. Theorem
\ref{th9} claims that if $G$ is K\"{o}nig-Egerv\'{a}ry, then $\varrho
_{e}\left(  G\right)  \leq m\left(  G\right)  -\xi\left(  G\right)
+\varepsilon\left(  G\right)  $. It justifies the following.

\begin{problem}
Bound $\varrho_{e}\left(  G\right)  $ using various graph invariants for
almost bipartite non-K\"{o}nig-Egerv\'{a}ry graphs.
\end{problem}

Clearly, $\left\vert V\left(  C\right)  \right\vert \leq\varrho_{e}\left(
G\right)  \leq$ $m\left(  G\right)  $ for every almost bipartite
non-K\"{o}nig-Egerv\'{a}ry graph $G$ with the unique odd cycle $C$.

\section{Declarations}

\textbf{Conflict of interest.} We declare that we have no conflict of interest.

\end{document}